\documentclass[reqno]{amsart}
\usepackage{latexsym,epsfig,amssymb,amsmath,amsthm,color,url}
\usepackage[comma,sort&compress]{natbib}
\allowdisplaybreaks 
 \numberwithin{equation}{section}

\newtheorem{thm}{Theorem}[section]
\newtheorem{lem}[thm]{Lemma}
\newtheorem{prop}[thm]{Proposition}
\newtheorem{cor}[thm]{Corollary}

\theoremstyle{definition}
\newtheorem{defn}[thm]{Definition}

\theoremstyle{remark}
\newtheorem{rem}[thm]{Remark}

\newcommand{\bC}{\boldsymbol{C}}

\newcommand{\bpi}{\boldsymbol{\pi}}
\newcommand{\bxi}{\boldsymbol{\xi}}
\newcommand{\bchi}{\boldsymbol{\chi}}

\newcommand{\bzeta}{\boldsymbol{\zeta}}
\newcommand{\brho}{\boldsymbol{\rho}}

\begin{document}
\bibliographystyle{plainnat}

\title{Strong Laws for Balanced Triangular Urns}

\author[A. Bose]{Arup Bose}
\address{Stat-Math Unit\\ Indian Statistical Institute\\ 203 B. T. Road\\ Kolkata 700108\\ India}
\email{abose@isical.ac.in}

\author[A. Dasgupta]{Amites Dasgupta}
\address{Stat-Math Unit\\ Indian Statistical Institute\\ 203 B. T. Road\\ Kolkata 700108\\ India}
\email{amites@isical.ac.in}

\author[K. Maulik]{Krishanu Maulik}
\address{Stat-Math Unit\\ Indian Statistical Institute\\ 203 B. T. Road\\ Kolkata 700108\\ India}
\email{krishanu@isical.ac.in}

\begin{abstract}
Consider an urn model whose replacement matrix is triangular, has all entries nonnegative and the row sums are all equal to one. We obtain the strong laws for the counts of balls corresponding to each color. The scalings for these laws depend on the diagonal elements of a rearranged replacement matrix. We use the strong laws obtained to study further behavior of certain three color urn models.
\end{abstract}

\keywords{Urn model, balanced triangular replacement matrix.}

\subjclass[2000]{Primary 60G70, 60F05; Secondary 60F10.}
\maketitle

\begin{section}{Introduction}
Consider an urn with balls of $(K+1)$ colors. Initially the counts of balls of each color are non-random, strictly positive real numbers and the total count of balls in the urn equals one. Let the row vector $\bC_0$ denote the initial count of balls of each color. The composition of the urn evolves by adding balls of different colors at times $n=1, 2, 3, \ldots$ as follows.

Suppose $R = ((r_{ij}))$ is a $(K+1) \times (K+1)$ non-random balanced (that is, each row sum is same and hence, without loss of generality, one) replacement matrix with nonnegative entries. Let $\bC_n$ denote the row vector of the counts of balls of each color after the $n$-th trial, $n= 1, 2, \ldots$. At the $n$-th trial, a ball is drawn at random from the urn with the current composition $\bC_{n-1}$, so that the $i$-th color appears with probability $\bC_{n-1, i}/n$, $ i=1, \ldots, (K+1)$.  If the $i$-th color appears, then, for $j=1, \ldots, (K+1)$, $r_{ij}$ balls of $j$-th color are added to the urn before the next draw, together with the drawn ball. It is of interest to study the stochastic behavior of $\bC_n$ as $n \to \infty$.

In case $R$ is irreducible, let $\bpi_R$ be the unique stationary distribution satisfying $\bpi_R R = \bpi_R.$ Then \citep[see, for example,][]{gouet:1997}
$\bC_n / (n+1) \to \bpi_R\  \text{almost surely}.$ Note that $\bpi_R$ is also a left eigenvector of $R$ corresponding to the eigenvalue $1$. However, when $R$ is not irreducible or balanced, the balls of different colors may increase at different rates and strong or weak limits for $\bC_n$ are not known in full generality.

\cite{janson:2006} considered two color triangular urn models, where the replacement matrix was not necessarily balanced, and identified the weak limits of $\bC_n$ in all possible cases. He mentioned urns with more colors and triangular replacement matrices as possible objects of further study \citep[cf.][Problem 1.16]{janson:2006}. \cite{flajolet:dumas:puyhaubert:2006} considered a three color urn having balanced triangular replacement matrix $R$ with further conditions on the entries and obtained weak limit theorems.

Motivated by these results, we consider balanced, triangular urns with arbitrary (but finite) number of colors. This assumption of balancedness on $R$ allows us convenient application of martingale techniques. In contrast, \cite{janson:2006} used the theory of branching processes and
\cite{flajolet:dumas:puyhaubert:2006} used generating functions. However, the application of martingale techniques to the study of urn models is not new, see for example, \cite{gouet:1997, bai:hu:1999}.

With appropriate scalings, we establish almost sure convergence of each color count to non zero limits. Under an additional assumption, see~\eqref{eq: unique arrangement}, the limits are expressed in terms of the limits of certain martingales and left eigenvectors of appropriate submatrices of $R$. These strong laws for urn models with arbitrary but finite number of colors and balanced triangular replacement matrices are the main contributions of this article.

The outline of the rest of the paper is as follows. Section~\ref{sec: notation} first describes a rearrangement of colors which converts any triangular balanced replacement matrix to an appropriate standard form. Our results are better described with reference to this standard form. Of course the convergence holds without assuming the standard form but the results are cumbersome to state, see Remark~\ref{rem: alternate}. We also state the additional assumption~\eqref{eq: unique arrangement} required to identify the limits in somewhat explicit forms. This section also establishes the notation to describe the limits and states some auxiliary results we need.

In Section~\ref{sec: result}, we state and prove the main theorem. In case of a color whose corresponding diagonal entry is larger than all the preceding ones, we consider the right eigenvector of $R$ corresponding to this eigenvalue and normalize the corresponding linear combination to obtain a martingale. This martingale turns out to be $L^2$-bounded and hence converges almost surely. The convergence of the individual color count then follows, since earlier colors have lower rates. For colors whose corresponding diagonal entry is not larger than the previous ones, we first show that the appropriately scaled color count is $L^1$-bounded. Then we form  the appropriate martingale and obtain the convergence. In Section~\ref{sec: three col}, we analyze the three color urn model with triangular replacement matrix as a corollary and obtain the asymptotic behaviors of linear combinations of color counts. This gives an indication of further results that can be proved using the strong laws of this article.
\end{section}

\begin{section}{Notation and Preliminary Results} \label{sec: notation}
Suppose $R$ is a balanced triangular replacement matrix with row
sums one. Denote the diagonal elements of $R$ as $r_k$, $1\leq k\leq K+1$. Let $1 = i_1 < i_2 < \cdots < i_J < i_{(J+1)} ( =K+1 )$ denote the indices of the
running maxima of the diagonals, namely, $r_1 = r_{i_1} \leq r_{i_2} \leq \cdots \leq r_{i_J} \leq r_{i_{(J+1)}} = r_{K+1}$ and for $i_j < k < i_{(j+1)}$, we have $r_k < r_{i_j}$ for $j=1, 2, \ldots, J$.

\begin{rem} \label{rem: block}
Since the row sums are $1$ and the elements of $R$ are nonnegative, all the diagonal elements will be less than or equal to $1$. Thus, $(K+1)$ will always be an index of the running maximum of the diagonals.
\end{rem}

The running maxima of $R$ also lead to the following concepts:
\begin{defn}
Suppose $R$ is a balanced triangular replacement matrix. For $j=1, 2, \ldots, J$, the colors indexed by $i_j, i_j+1, \ldots, i_{(j+1)} -1$ constitute the $j$-th \textit{block} of colors, $i_j$ is called its \textit{leading index}, and the corresponding color is called the \textit{leading color} of the $j$-th block. The color indexed by $(K+1)$ will be the leading color and the sole constituent of the $(J+1)$-st block.
\end{defn}

The triangular replacement matrix $R$ with the indices of the running maxima of the diagonals can be visualized as
\begin{equation*}
R = \left(
\begin{array}{cccccccccccccc}
r_1 &r_{12} &\cdots &\cdots &\cdots &\cdots &\cdots &\cdots &\cdots &\cdots &\cdots &\cdots &\cdots &\cdots \\
&r_2 &\cdots &\cdots &\cdots &\cdots &\cdots &\cdots &\cdots &\cdots &\cdots &\cdots &\cdots &\cdots \\
&&\ddots &\cdots &\cdots &\cdots &\cdots &\cdots &\cdots &\cdots &\cdots &\cdots &\cdots &\cdots \\
&&&r_{i_2-1} &\cdots &\cdots &\cdots &\cdots &\cdots &\cdots &\cdots &\cdots &\cdots &\cdots \\
&&&&r_{i_2} &\cdots &\cdots &\cdots &\cdots &\cdots &\cdots &\cdots &\cdots &\cdots \\
&&&&&\ddots &\cdots &\cdots &\cdots &\cdots &\cdots &\cdots &\cdots &\cdots \\
&&&&&&r_{i_j} &\cdots &r_{i_j, k} &\cdots &r_{i_j, i_{(j+1)} - 1} &r_{i_j, i_{(j+1)}} &\cdots &\cdots \\
&&&&&&&\ddots &\cdots &\cdots &\cdots &\cdots &\cdots &\cdots \\
&&&&&&&&r_k &\cdots &r_{k, i_{(j+1)} - 1} &r_{k, i_{(j+1)}} &\cdots &\cdots \\
&&&&&&&&&\ddots &\cdots &\cdots &\cdots &\cdots \\
&&&&&&&&&&r_{i_{(j+1)}-1} &r_{i_{(j+1)}-1, i_{(j+1)}} &\cdots &\cdots \\
&&&&&&&&&&&r_{i_{(j+1)}} &\cdots &\cdots \\
&&&&&&&&&&&&\ddots &\cdots \\
&&&&&&&&&&&&&r_{i_{(J+1)}}
\end{array}
\right)
\end{equation*}
Here $r_1 = r_{i_1} \leq r_{i_2} \leq \cdots \leq r_{i_j} \leq
r_{i_{(j+1)}} = \leq \cdots \leq r_{i_{(J+1)}}$ give the running
maxima of the diagonal entries. It will be helpful to study the concepts of rearrangement and blocks, while keeping this visualization in mind.

To study urn models with triangular replacement matrices, we need to arrange the colors systematically, which we describe next. This particular rearrangement keeps the replacement matrix triangular. The new replacement matrix is obtained  by pre- and post-multiplication of $R$ by permutation matrices. Thus, it remains balanced with row sum $1$ and has the same set of
eigenvalues. The elements of the new eigenvectors are also suitable
rearrangements of the original ones.

\begin{defn}
The colors are said to be arranged in the \textit{increasing} order if $R$ satisfies the following: with $1 = i_1 < i_2 < \cdots < i_J < i_{(J+1)} ( =K+1 )$ as the indices of the running maxima of the diagonals, for $i_j < k < i_{(j+1)}$, $j=1, 2, \ldots, J$, we have
\begin{equation} \label{eq: arrangement}
\sum_{m=i_j}^{k-1} r_{mk} > 0.
\end{equation}
\end{defn}

It is easy to see that,~\eqref{eq: arrangement} is equivalent to the fact that, for any non-leading color with index $k$ in $j$-th block, namely, for $i_j < k < i_{(j+1)}$, with $j=1, 2, \ldots, J$, the part of $k$-th column in the $j$-th block has at least one non-zero entry. Also note that Condition~\eqref{eq: arrangement} holds only for non-leading colors.

The next proposition shows that any urn model with triangular replacement matrix can be transformed into another urn model with triangular replacement matrix such that the colors are in increasing order.

\begin{prop} \label{prop: arrangement}
Suppose $R$ is a balanced triangular replacement matrix with row sums one. Then there exists a rearrangement of the colors in the increasing order, such that the replacement matrix remains triangular.
\end{prop}

\begin{proof}
From Remark~\ref{rem: block}, we have that the $(K+1)$-st color forms the last block as required. We shall now construct the other blocks inductively going backward. Within a block, the construction will move forward.

Suppose we have constructed some blocks. If the leading index of the last constructed block, say $i$, is 1, we are done.

If $i > 1$, we construct the next (previous) block as follows. Let $k <i$ be such that $r_k= \max\{r_m: \ m<i\}$. The color with index $k$ is declared to be the leading color of the present block under construction and the index of this color may change, as discussed later, through rearrangement.

By our choice of leading colors, the diagonal entries of the indices of the leading colors will be in nondecreasing order, as required.

Next we decide which of the intermediate colors with index $m$, $k<m<i$ will be in the present block. This will be done through a process of rearrangement described inductively going forward.

Suppose $l$ colors, including the leading color, satisfying~\eqref{eq: arrangement}, have been obtained through rearrangement for the present block and the index of the leading color has changed to $k' (>k)$ after this rearrangement. Then the index of the last considered color was $k'+l-1$. If $k'+l=i$, we have considered all intermediate colors and the construction of the block is over.

If $k'+l <i$, consider the color with index $k'+l$. By choice of the leading color of the present block, we must have $r_{k'} > r_{k'+l}$. If we have $\sum_{m=k'}^{k'+l-1} r_{m, (k'+l)} > 0$, we take the color with index $k'+l$ as the $(l+1)$-st color of the present block.

Otherwise, $r_{m, (k'+l)} = 0$ for $m = k', k'+1, \ldots, k'+l-1$. In that case, we reshuffle the colors to bring $(k'+l)$-th color ahead of the $k'$-th one, and then $r_{m, (k'+l)}$, $m = k', k'+1, \ldots, k'+l-1$ will be the only entries which will move below the diagonal of the $k^\prime$-th column in the reshuffled replacement matrix. Hence the reshuffled replacement matrix will remain triangular. After reshuffle, this color will have index $k'$ and the index of the colors already in the present block will increase by $1$, with the present leading index increasing to $k'+1$. The number of colors in the present block will remain at $l$. This gives the forward induction step for constructing a block. Since a color is shuffled up if it fails~\eqref{eq: arrangement}, all the remaining colors will satisfy this condition. Thus we complete the backward induction step for rearrangement of blocks.
\end{proof}

In view of the above proposition, we shall always assume, unless otherwise
mentioned, that the colors are indeed in increasing order.

Note that if $r_{i_j} = r_{i_{j+1}}$ and $r_{m, i_{(j+1)}} = 0$ for all $m = i_j, i_j+1, \ldots, i_{(j+1)} -1$, then we can reshuffle the colors to bring the $i_{(j+1)}$-th color ahead of the $i_j$-th one, yet maintaining the triangular structure of the replacement matrix and the increasing order of the colors. Hence the rearrangement of the colors in the increasing order will not be unique. To make the above rearrangement of the colors to the increasing order a unique one, we further assume that
\begin{equation} \label{eq: unique arrangement}
\sum_{m=i_{j}}^{i_{(j+1)}-1} r_{m,i_{(j+1)}} > 0, \text{ whenever $r_{i_{j}} = r_{i_{(j+1)}}.$}
\end{equation}
Assumption~\eqref{eq: unique arrangement} is equivalent to requiring at least one non-zero entry in the $i_{(j+1)}$-th column in the part corresponding to the $j$-th block. Its significance has been discussed later in Remark~\ref{rem: nonzero}.

We define the following submatrices and vectors corresponding to different blocks of colors.
\begin{defn}
Let $R^{(j)}$ be the submatrix formed by the rows and columns corresponding to the indices of the $j$-th block. We shall also write $\lambda_j = r_{i_j} = r^{(j)}_1$. The part of the vector $\bC_n$ corresponding to the $j$-th block will be denoted by $\bC^{(j)}_n$.  Finally, $\brho^{(j)}$ will denote the part of $i_{j+1}$-th column corresponding to the $j$-th block.
\end{defn}

By the definition of a block, $r_{i_j}= \lambda_j$ is the strictly largest
eigenvalue of $R^{(j)}$ and has multiplicity $1$. Let $\bpi^{(j)}$ be the unique left eigenvector of $R^{(j)}$ corresponding to the eigenvalue $\lambda_j$ normalized so that its first element is $1$. Then $\bpi^{(j)}$ satisfies
$$\bpi^{(j)}R^{(j)}=\lambda_j \bpi^{(j)}, \ \ \bpi^{(j)}_{1}=1.$$

Observe that if $i_j \leq m, k < i_{(j+1)}$, then $\bC_{nm} = \bC^{(j)}_{n, (m-i_j+1)}$ and $r_{mk} = r^{(j)}_{(m-i_j+1), (k-i_j+1)}$.

Next we define an index to count the number of times the diagonal entry corresponding to a leading color has occurred before. This is important in obtaining the rate of color count in Theorem~\ref{thm: main}.
\begin{defn}
For the $j$-th block with leading color index $i_j$, let
\begin{equation*}
\nu_j=\#\{m: \ r_m=\lambda_j, \ m<i_j\}.
\end{equation*}
\end{defn}
Observe that, if $r_m = \lambda_j$, for some $m<i_j$, then $m$ is a leading index as well. So, if it is the first time a diagonal has value $\lambda_j$, we have $\nu_{j} = 0$. Also note that $\lambda_{j-1} = \lambda_j$ if and only if $\nu_j > 0$, and in that case $\nu_{j-1} = \nu_j - 1$ holds.

The following useful result is obtained as a consequence of the above definitions.
\begin{lem} \label{lem: matrix}
If the colors are in increasing order and the replacement matrix $R$ is triangular, then the vector $\bpi^{(j)}$ has all coordinates positive.
\end{lem}
\begin{proof}
We prove this by induction on the coordinates of the vector.

Observe that if $r_{i_j}=0$, then the $j$-th block has only one color, namely, the $i_j$-th one. By the choice of normalization, $\bpi^{(j)}_{1} = 1$. If $r_{i_j}=0$, then by the above observation, the proof is complete. So, without loss of generality, we can take $r_{i_j}>0$.

Now assume that the first $k (< i_{(j+1)} - i_j)$ coordinates of $\bpi^{(j)}$ are positive. By the property of the eigenvector of $R^{(j)}$, we
have $\sum_{m=1}^{k+1} \bpi^{(j)}_{m} r_{(m+i_j-1), (k+i_j)} = r_{i_j} \bpi^{(j)}_{k+1}$, which gives
\begin{equation*}
\bpi^{(j)}_{k+1} = \frac1{r_{i_j} - r_{k+i_j}} {\sum_{m=1}^k \bpi^{(j)}_{m} r_{(m+i_j-1), (k+i_j)}}.
\end{equation*}
Now the denominator on the right side is positive, since $r_{i_j}$ is the strictly largest eigenvalue. By induction hypothesis, $\bpi^{(j)}_{m} > 0$, for $m = 1, \ldots, k$. Also, by~\eqref{eq: arrangement}, $r_{(m+i_j-1), (k+i_j)} > 0$ for some $m = 1, \ldots, k$. This proves the induction step and the lemma.
\end{proof}

We also denote $\Pi_n (s) = \prod_{i=0}^{n-1} \left( 1+\frac{s}{i+1} \right)$. Recall that Euler's formula for gamma function gives
\begin{equation} \label{eq: euler}
\Pi_n (s) \sim n^s / \Gamma(s+1), \text{ if $s$ is not a negative integer}.
\end{equation}
\end{section}

\begin{section}{Main Results} \label{sec: result}
Now we state and prove our main result on the strong convergence of individual color counts.

\begin{thm} \label{thm: main}
Suppose $R$ is a $(K+1) \times (K+1)$ balanced and triangular with row sums one and $(J+1)$ blocks and the colors are in increasing order satisfying~\eqref{eq: unique arrangement}. Then, for $j=1, 2, \ldots, J+1$,
\begin{equation*}
\bC^{(j)}_{N} / \{N^{\lambda_j} (\log N)^{\nu_{j}}\} \to \bpi^{(j)}
V_j
 \ \text{almost surely, as well as in} \ L^2,
\end{equation*}
where $V_{J+1} = 1$. If $r_1=0$, then $V_1 = \bC_{01}$. If $r_1>0$, then $V_1$ is a non-degenerate random variable. For $j = 2, 3, \ldots, J$, if $\nu_{j} = 0$, then $V_j$ is also some non-degenerate random variable. If $\nu_{j} > 0$, we further have
\begin{equation} \label{eq: limit reln}
V_j = \frac1{\nu_{j}} \bpi^{(j-1)} \brho^{(j-1)} V_{j-1}.
\end{equation}
\end{thm}

\begin{rem} \label{rem: nonzero}
Note that, due to Lemma~\ref{lem: matrix} $\bpi^{(j-1)}$ has all entries positive, and $\brho^{(j-1)}$ has at least one entry positive by~\eqref{eq: unique arrangement}, which makes $\bpi^{(j-1)} \brho^{(j-1)}>0$ and hence, recursively, all the $V_j$'s are nondegenerate.
\end{rem}

\begin{rem} \label{rem: diag 1}
If $r_j=1$ for some $j \leq K$, then the rest of the entries in the $j$-th row are zero. Thus,~\eqref{eq: arrangement} requires $r_{j+1}=1$ and the $(j+1)$-st color will be the leading color of a new block. However, if~\eqref{eq: unique arrangement} is also assumed, even this is not possible. So, in the setup of Theorem~\ref{thm: main} above, we must have $r_j < 1$ for all $j \leq K$.
\end{rem}

\begin{rem} \label{rem: alternate}
The rearrangement of the colors to the increasing order and Condition~\eqref{eq: unique arrangement} help us identify the limits in Theorem~\ref{thm: main}. However, it will be clear from the proof that even without this assumption, appropriate strong laws hold. In this approach, we do not use the concept of blocks. It can be shown that
$$\frac{1}{N^{r_1}} \bC_{N1} \to W_1 , \ \text{almost surely}$$
for some random variable $W_1$. We can then define inductively, the rates for all colors $j > 1$ as follows: assume that for all $1 \leq j \leq k$, there exists $s_j$ and $\delta_j$ and random variables $W_j$ such that
$$\frac1{ N^{s_j} (\log N)^{\delta_{j}}} \bC_{Nj} \to W_j  \ \text{almost surely.}$$
If the part of the $(k+1)$-st column above the diagonal has all entries $0$, then for some random variable $W_{k+1}$,
$$\frac1{N^{r_{(k+1)}}} \bC_{N, (k+1)} \to W_{k+1}.$$

On the other hand, suppose $r_{j, (k+1)}>0$ for some $j=1, 2, \ldots, k$. Consider all the colors indexed by $j$ such that $r_{j, (k+1)}>0$. Let the highest rate of convergence for such color counts be $n^s (\log n)^\delta$. Then one can say that
$$\frac1{a_n} \bC_{N, (k+1)} \to W_{k+1}\ \text{almost surely,}$$ for some random variable $W_{k+1}$ where
\begin{equation*}
a_n =
\begin{cases}
n^s (\log n)^\delta, &\text{if $r_{k+1} < s,$}\\
n^s (\log n)^{\delta+1}, &\text{if $r_{k+1} = s,$}\\
n^{r_{k+1}}, &\text{if $r_{k+1} > s$}.
\end{cases}
\end{equation*}
It is not clear if one can easily write down how the $W_k$'s are related. However, under Condition~\eqref{eq: unique arrangement}, if we rearrange the colors in the increasing order, the subvectors corresponding to each block obtained from the rearranged $W_k$'s are of course the same as $\bpi^{(j)} V_j$'s and will satisfy~\eqref{eq: limit reln}.
\end{rem}

\begin{proof}[Proof of Theorem~\ref{thm: main}]
The proof is through induction on the index of color $l$. Let $\bchi_n$ be the row vector of order $(K+1)$ with the $m$-th entry 1 if the $m$-th color is drawn at the $n$-th draw, the other entries being zero. Let $\mathcal{F}_n$ denote the $\sigma$-field generated by $\{ \bchi_k : 1 \leq k \leq n\}$.

So we first quickly verify the result for $l=1$.  If $r_1 = 0$, then the entire first column is zero, so the first color count cannot change whichever color be drawn. Thus $\bC_{n1}$ stays constant at $\bC_{01}$ and the result is trivially true. Next consider $r_1>0$. In that case, we pool all the remaining colors and that gives us the replacement matrix
\begin{equation*}
\begin{pmatrix}
  r_1 & 1 - r_1\\
  0 & 1
\end{pmatrix}.
\end{equation*}
Then the result  for $l=1$ follows from Proposition 2.2 (iii) of
\cite{bose:dasgupta:maulik:2008}.

Now assume that the result holds for the first $(l-1)$ colors for some $l \geq 2$. Suppose the next color is the $k$-th color of the $j$-th block. Then we have $l=i_j+k-1$.

The following two observations follow from the induction hypothesis:

(i) If $l$ is a leading color, that is, $k=1$ and $l=i_j$, we have
\begin{equation} \label{eq: induction}
\frac{\bC_{Nm}}{N^{\lambda_{j-1}} (\log N)^{\nu_{{j-1}}}} \to
\begin{cases}
\bpi^{(j-1)}_{m+1-i_{(j-1)}} V_{j-1}, &\text{if $i_{(j-1)} \leq m
< i_j$},\\
0, &\text{if $m < i_{(j-1)}$}
\end{cases}
\end{equation}
almost surely, as well as, in $L^2$.

(ii) If $l$ is not a leading color, that is, $k>1$ and $l>i_j$, we
have
\begin{equation} \label{eq: induction not lead}
\frac{\bC_{Nm}}{N^{\lambda_{j}} (\log N)^{\nu_j}} \to
\begin{cases}
\bpi^{(j)}_{m+1-i_j} V_{j}, &\text{if $i_j \leq m < l$},\\
0, &\text{if $m < i_j$}
\end{cases}
\end{equation}
almost surely, as well as, in $L^2$. In particular, we have, for $m<l$,
\begin{equation} \label{eq: induction moment}
E[\bC_{Nm}] =
\begin{cases}
O(N^{\lambda_{j-1}} (\log N)^{\nu_{j-1}}), &\text{if $k=1$},\\
O(N^{\lambda_j} (\log N)^{\nu_j}), &\text{if $k>1$}.
\end{cases}
\end{equation}

We separate the proof into three cases: Case 1:  $k=1$, $\nu_j = 0$; Case 2:  $k=1$, $\nu_j>0$ and Case 3: $k>1$.

\textbf{Case 1: $k=1$ and $\nu_j = 0$.} Let $\bzeta$ be a right eigenvector of $R$ for the eigenvalue $\lambda_{j} = r_{i_{j}}$, normalized so that its $\bzeta_{i_{j}} = 1$. Observe that $\bzeta_k = 0$ for $k> i_{j}$, and since $\nu_j = 0$ gives $r_k < r_{i_{j}}$ for all $k<i_{j}$, we have $\bzeta_k \geq 0$, for $k\leq i_{j}$.

Consider the martingale $U_N = \bC_N \bzeta / \Pi_N
(\lambda_{j})$. Then the martingale difference is
\begin{equation*}
U_{N+1} - U_N = \frac{\lambda_{j}}{\Pi_{N+1} (\lambda_{j})}
\left(\bchi_{{}_{N+1}} - \frac{\bC_N}{N+1} \right) \bzeta.
\end{equation*}
Denote $\bzeta^2$ to be the column vector whose coordinates are squares of those of $\bzeta$. Hence, we have
\begin{multline*}
E[(U_{N+1} - U_N)^2 | \mathcal F_N] =
\frac{\lambda_{j}^2}{(\Pi_{N+1} (\lambda_{j}))^2} \left[ \frac{\bC_N
\bzeta^2}{N+1} - \left( \frac{\bC_N \bzeta}{N+1} \right)^2 \right]\\
\leq \frac{1}{(\Pi_{N+1} (\lambda_{j}))^2} \frac{\bC_N \bzeta^2}{N+1} \leq \frac{\|\bzeta \|_\infty}{(N+1) \Pi_{N+1}(\lambda_j)} U_N \leq \| \bzeta \|_\infty \Gamma(\lambda_j+1) \frac{1+U_N^2}{(N+1)^{1+\lambda_{j}}},
\end{multline*}
for all large enough $N$, where $\| \bzeta \|_\infty$ is the largest coordinate of $\bzeta$ (recall that all the coordinates of $\bzeta$ are non-negative) and the last inequality follows using the fact $2 U_N \leq 1+U_N^2$ and~\eqref{eq: euler}.

As in the proof of Proposition 2.2 (iii) of~\cite{bose:dasgupta:maulik:2008}, this gives an iteration for $1 + E[U_N^2]$ and we can prove $U_N$ is $L^2$-bounded and hence converges almost surely, as well as, in $L^2$. Thus,
by~\eqref{eq: euler}, $\bC_N \bzeta / N^{\lambda_j}$ also converges almost surely, as well as, in $L^2$, to $V_j$, say.

Note that $U_1 = \bC_1 \bzeta / (1+\lambda_j) = (\bC_0 \bzeta + \lambda_j \bchi_1 \bzeta)/ (1+\lambda_j)$. Since $\bC_0$ has all coordinates positive, $\bchi_1$ takes all coordinate vectors as values with positive probability. Thus, $\bchi_1 \bzeta$ is constant if and only if $\bzeta$ has all coordinates of same value. This will be the case if and only if the corresponding eigenvalue is $1$, which, by Remark~\ref{rem: diag 1}, holds if and only if $l=K+1$ and $j=J+1$. So in that case, $\bC_N \bzeta = N+1$ and we have $U_N =
\bC_N \bzeta / (N+1) = 1 = V_{J+1}$.

If $j\leq J$, $U_1$ is nondegenerate and hence has positive variance. Since $U_N$ is a martingale, variance of $U_N$ is nondecreasing and the limit variable has nonzero variance. So the limit variable $V_j$ is nondegenerate for $j\leq J$.

Finally using the limit of $\bC_N \bzeta / N^{\lambda_{j}}$, since $\bzeta_k = 0$ for $k>i_j$ and $\bzeta_{i_j} = 1$, we have
$$\lim_{N\to\infty} \frac1{N^{\lambda_j}} \bC_{N1}^{(j)} = V_j - \sum_{m=1}^{i_j-1} \bzeta_m \lim_{N\to\infty} \frac1{N^{\lambda_j}} \bC_{Nm} \ \text{almost surely and in}\  L^2,$$ provided the limits on the right side exist. Now, since $\nu_j = 0$, we have $\lambda_{j-1} < \lambda_j$, and by~\eqref{eq: induction}, the limits on the right side are all zero. Thus,
$$\frac1{N^{\lambda_j}} \bC_{N1}^{(j)} \to V_j\ \text{almost surely and in}\  L^2 ,$$ and $V_j$ is nondegenerate for $j\leq J$, $V_{J+1}=1$. Since $\bpi^{(j)}_1 = 1$, we have proved the induction step for $k=1$ and $\nu_j=0$ for Case 1.

For the remainder of the cases, the proof is done in two steps. We first show $L^1$-boundededness of $Z_N := \bC^{(j)}_{Nk} / ( N^{\lambda_{j}} (\log N)^{\nu_j} )$ and then we show the required almost sure and $L^2$-convergence by constructing an appropriate martingale.

\textit{Step 1 ($L^1$ bound)}\,: Observe that $\bC_{N+1, l} = \bC_{Nl} + \sum_{m=1}^{l} \bchi_{{}_{N+1, m}} r_{ml},$ which gives,
\begin{align}
E \left[\bC_{N+1, l} | \mathcal{F}_N\right] &= \bC_{Nl} \left( 1 +
\frac{r_l}{N+1} \right) + \frac1{N+1} \sum_{m=1}^{l-1} \bC_{N m} r_{ml}, \label{eq: cond exp lead rep}
\intertext{leading to}
E \left[\bC_{N+1, l}\right] &= E [\bC_{Nl}] \left( 1 + \frac{r_l}{N+1} \right) + \frac1{N+1} \sum_{m=1}^{l-1} E [\bC_{N m}] r_{ml} \nonumber
\intertext{Iterating, we have}
&= \bC_{0l} \Pi_N(r_l) + \sum_{m=1}^{l-1} r_{ml} \sum_{n=0}^N \frac1{n+1} E [\bC_{nm}] \frac{\Pi_N(r_l)}{\Pi_n(r_l)}
\intertext{to conclude, using $\bC_{0l} = \bC^{(j)}_{0k}$ and $\bC_{Nl} = \bC^{(j)}_{Nk}$,}
\frac{E \left[ \bC^{(j)}_{Nk} \right]}{\Pi_N(r_l)} &= \bC^{(j)}_{0k} + \sum_{m=1}^{l-1} r_{ml} \sum_{n=0}^{N-1} \frac1{n+1} \frac{E[\bC_{nm}]}{\Pi_n(r_l)}. \label{eq: exp lead rep}
\end{align}

\textbf{Case 2: $k=1$ and $\nu_j>0$}. We also have $\nu_{j-1} = \nu_j - 1$ and $r_l = \lambda_j = \lambda_{j-1}$. Then using~\eqref{eq: euler},~\eqref{eq: induction moment} and~\eqref{eq: exp lead rep}, we have
\begin{equation*}
\frac{E\left[ \bC_{Nk}^{(j)} \right]}{\Pi_{N}(\lambda_j)} = \bC^{(j)}_{0k} + \sum_{m=1}^{l-1} r_{ml} \sum_{n=0}^{N-1} \frac{(\log (n+2))^{\nu_{j - 1}}}{n+1} \frac{E[\bC_{nm}]}{n^{\lambda_j} (\log (n+2))^{\nu_j - 1}} \frac{n^{r_l}}{\Pi_n(r_l)} = O\left((\log n)^{\nu_j}\right).
\end{equation*}
Thus, again using~\eqref{eq: euler}, $\{Z_N\}$ becomes $L^1$-bounded.

\textbf{Case 3: $k>1$}. Here we have $r_l < \lambda_j$. Then using~\eqref{eq: euler},~\eqref{eq: induction moment} and~\eqref{eq: exp lead rep}, we have
\begin{equation*}
\frac{E\left[ \bC_{Nk}^{(j)} \right]}{\Pi_{N}(r_l)} = \bC^{(j)}_{0k} + \sum_{m=1}^{l-1} r_{ml} \sum_{n=0}^{N-1} \frac{(\log (n+2))^{\nu_j}}{(n+1) n^{r_l - \lambda_j}} \frac{E[\bC_{nm}]}{n^{\lambda_j} (\log (n+2))^{\nu_j}} \frac{n^{r_l}}{\Pi_n(r_l)} = O\left(n^{\lambda_j - r_l} (\log n)^{\nu_j}\right).
\end{equation*}
Thus, again using~\eqref{eq: euler}, $\{Z_N\}$ becomes $L^1$-bounded.

\textit{Step 2 (Convergence)}\,: Now we construct the relevant martingale. Using~\eqref{eq: cond exp lead rep}, it is easy to check that
\begin{equation} \label{eq: def mg lead rep}
M_N = \frac{\bC_{Nl}}{\Pi_N(r_l)} - \sum_{m=1}^{l-1} \sum_{n=0}^{N-1} \frac{r_{ml}}{n+1+r_l} \frac{\bC_{nm}}{\Pi_n(r_l)}
\end{equation}
forms a martingale. The martingale difference is given by
\begin{equation*}
M_{N+1} - M_N = \frac1{\Pi_{N+1}(r_l)} \sum_{m=1}^l \left( \bchi_{{}_{(N+1),m}} - \frac{\bC_{Nm}}{N+1} \right) r_{ml}
\end{equation*}
which leads to, using~\eqref{eq: induction moment} and $L^1$-boundedness of $Z_N$,
\begin{align}
E[ (M_{N+1} - M_N)^2 ] &= \frac1{(\Pi_{N+1}(r_l))^2} E\left[ \frac{\sum_{m=1}^l \bC_{Nm} r_{ml}^2}{N+1} - \left( \frac{\sum_{m=1}^l \bC_{Nm} r_{ml}}{N+1}
\right)^2 \right] \nonumber\\
&\leq \frac1{(\Pi_{N+1}(r_l))^2} E\left[\frac{\sum_{m=1}^l \bC_{Nm}
r_{ml}^2}{N+1}\right] \label{eq: sq mg diff inter}\\
&= O \left( \frac{(\log N)^{\nu_j}}{N^{1+2r_l-\lambda_j}} \right). \label{eq: sq mg diff}
\end{align}

When $k=1$, $\nu_j>0$ and $\lambda_j = r_l = 0$, we can further improve on the order of the squared moment of the martingale difference given in~\eqref{eq: sq mg diff}. First, observe that $l$ being a leading color and $r_l =0$ imply that $r_m = 0$ for all $m < l$, which makes each of the colors indexed by $m \leq l$ a leading color of a block of size $1$. This implies $j=l$. Since the diagonal elements corresponding to all these colors are $0$, we have $\nu_m = m-1$ for $m \leq l$. Since $\nu_j = j-1 = l-1 > 0$, we have $l\geq 2$. Thus,~\eqref{eq: induction moment} simplifies to $E[\bC_{nm}] = O \left( (\log n)^{m-1} \right)$, for $m<l$. Also, $r_l$ being zero, the $l$-th term in the sum of~\eqref{eq: sq mg diff inter} does not contribute. Hence, we have
\begin{equation} \label{eq: sq mg diff lead zero}
E[ (M_{N+1} - M_N)^2 ] = O \left( \frac{(\log N)^{(\nu_j-1)}}{N} \right).
\end{equation}

\textbf{Case 2: $k=1$ and $\nu_j>0$.} Here we have $l=i_j$ and $r_l = \lambda_j$. First assume that $r_l=\lambda_j>0$. Then the right side of~\eqref{eq: sq mg diff} is summable. Hence $M_N$ is an $L^2$-bounded martingale, which converges almost surely, as well as in $L^2$. Since $\nu_j
> 0$, we have
$$\frac1{(\log N)^{{\nu_j}}} M_N \to 0 \ \text{almost surely and in} \ L^2.$$

Next assume $r_l=\lambda_j=0$. Then, using~\eqref{eq: sq mg diff lead zero}, we have $M_N/(\log N)^{\nu_j/2}$ is $L^2$-bounded. Hence, we have $M_N/(\log N)^{\nu_j} \to 0$ in $L^2$. We shall now show that
$$Y_N := \frac1{(\log N)^{\nu_j}} M_N\ \text{converges almost surely.}$$
Since the $L^2$-limit is known to be $0$, we shall then have $Y_N \to 0$ almost surely as well as in $L^2$. With $\Delta_N = 1/(\log N)^{\nu_j}$, we have $Y_N = M_N \Delta_N$, which gives
\begin{equation} \label{eq: Z diff lead zero}
Y_{N+1} - Y_N = M_{N+1} (\Delta_{N+1} - \Delta_N) + \Delta_N (M_{N+1} - M_N).
\end{equation}
Thus, it is enough to show that the partial sums of each of the terms on the right side of~\eqref{eq: Z diff lead zero} converges almost surely.

Now, $\Delta_N$ is a deterministic sequence and
\begin{equation*}
|\Delta_{N+1} - \Delta_N| = \Delta_{N+1} \left[ \left( \frac{\log (N+1)}{\log N} \right)^{\nu_j} - 1 \right] \sim \frac{\nu_j}{N (\log N)^{\nu_j + 1}}.
\end{equation*}
We further know that $M_N / (\log N)^{\nu_j/2}$ is $L^2$-bounded and hence $L^1$-bounded. Thus, recalling $\nu_j>0$, we have $E [|M_{N+1} (\Delta_{N+1} - \Delta_N)|]$ is summable and hence the first term on the right side of~\eqref{eq: Z diff lead zero} is almost surely absolutely summable.

Now, since $M_{N+1} - M_N$ is a martingale difference, so is the second term on the right side of~\eqref{eq: Z diff lead zero} as well. By~\eqref{eq: sq mg diff lead zero}, we have
\begin{equation*}
E [\Delta_N^2 (M_{N+1} - M_N)^2] = O\left(\frac1{(\log N)^{2 \nu_j}} \frac{(\log N)^{\nu_j -1}}{N} \right),
\end{equation*}
which is summable, as $\nu_j>0$. Hence the second term on the right side of~\eqref{eq: Z diff lead zero} is the difference sequence of a martingale which converges almost surely, as well as in $L^2$. Thus we obtain that $Y_N$ converges almost surely, as well as, in $L^2$ to 0 even when $r_l = \lambda_j = 0$.

Hence, under the assumption $k=1$ and $\nu_j>0$,
$$\frac1{(\log N)^{\nu_j}} M_N \to 0 \ \text{almost surely and in} \ L^2.$$

Using~\eqref{eq: def mg lead rep}, we then have
\begin{equation*}
\lim_{N\to\infty} \frac{\bC_{N, i_j}}{N^{\lambda_{j}} (\log N)^{\nu_{j}}} = \lim_{N\to\infty} \frac{\Pi_N(\lambda_j)}{N^{\lambda_{j}} (\log N)^{\nu_{j}}}
\sum_{m=1}^{i_j-1} \sum_{n=0}^{N-1} r_{m, i_j} \frac{(\log n)^{\nu_{j}-1}}{n+1+\lambda_{j}} \frac{n^{\lambda_{j}}}{\Pi_n(\lambda_{j})}
\frac{\bC_{nm}}{n^{\lambda_{j}} (\log n)^{\nu_{j}-1}},
\end{equation*}
where the limit is in the almost sure as well as $L^2$ sense. Since $\nu_j>0$, we have $\lambda_j = \lambda_{j-1}$ and $\nu_j - 1 = \nu_{j-1}$. Thus, from~\eqref{eq: euler} and~\eqref{eq: induction},
we have
\begin{equation} \label{eq: vj one}
\frac{\bC^{(j)}_{N1}}{N^{\lambda_{j}} (\log N)^{\nu_{j}}} = \frac{\bC_{Ni_j}}{N^{\lambda_{j}} (\log N)^{\nu_{j}}} \to \frac1{\nu_{j}} \sum_{m=1}^{i_j - i_{(j-1)}} \bpi^{(j-1)}_{m} \brho^{(j-1)}_{m} V_{j-1} = \frac1{\nu_j} \bpi^{(j-1)} \brho^{(j-1)} V_{j-1}
\end{equation}
almost surely, as well as, in $L^2$. We obtain the formula of $V_{j}$ in terms of $V_{j-1}$ from~\eqref{eq: vj one}. Since, by normalization, $\bpi^{(j)}_1 = 1$, we have the induction step for $k=1$, $\nu_j>0$.

\textbf{Case 3: $k>1$.} Thus, $r_l<\lambda_j$ holds and hence $\lambda_j > 0$. If $r_l > \lambda_j/2$, using~\eqref{eq: sq mg diff}, we have $M_N$ is an $L^2$-bounded martingale and hence converges almost surely as well as in $L^2$. Thus
$$M_N/\{N^{\lambda_j - r_l} (\log N)^{\nu_j}\} \to 0 \ \text{almost surely and in} \ L^2.$$

The analysis is a bit more elaborate when $r_l \leq \lambda_j/2$. If $r_l = \lambda_j/2$, using~\eqref{eq: sq mg diff}, $M_N / (\log N)^{(\nu_j+1)/2}$ is $L^2$-bounded. On the other hand, if $r_l < \lambda_j/2$, again using~\eqref{eq: sq mg diff}, $M_N / \left(N^{\lambda_j/2 - r_l} (\log N)^{\nu_j/2}\right)$ is $L^2$-bounded. Hence, for $r_l \leq \lambda_j/2$,
$$M_N/\{N^{\lambda_j - r_l} (\log N)^{\nu_j}\} \to 0\ \text{in}\  L^2.$$

We shall now show that $Y_N := M_N/\{N^{\lambda_j - r_l} (\log N)^{\nu_j}\}$ converges almost surely (and hence to 0) even when $r_l \leq \lambda_j/2$. With $\Delta_N = 1/\{N^{\lambda_j - r_l} (\log N)^{\nu_j}\}$, we have $Y_N = M_N \Delta_N$, which gives
\begin{equation*}
Y_{N+1} - Y_N = M_{N+1} (\Delta_{N+1} - \Delta_N) + \Delta_N (M_{N+1} - M_N).
\end{equation*}
As before, it is enough to show that the partial sums of each of the terms on the right side converges almost surely, which can be proved in a similar manner, using
\begin{equation*}
|\Delta_{N+1} - \Delta_N| = \Delta_{N+1} \left[ \left( 1+\frac1N \right)^{\lambda_j - r_l} \left( \frac{\log (N+1)}{\log N}
\right)^{\nu_j} - 1 \right] \sim \frac{\lambda_j - r_l}{N^{1 +
\lambda_j - r_l} (\log N)^{\nu_j}}.
\end{equation*}
We leave the details for the reader.

Using~\eqref{eq: def mg lead rep}, we have
\begin{equation*}
\lim_{N\to\infty} \frac{\bC_{Nl}}{N^{\lambda_{j}} (\log N)^{\nu_{j}}} = \lim_{N\to\infty} \frac{\Pi_N(r_l)}{N^{\lambda_j} (\log N)^{\nu_{j}}} \sum_{m=1}^{l-1} \sum_{n=0}^{N-1} r_{ml} \frac{(\log n)^{\nu_{j}}}{(n+1+r_l) n^{r_l - \lambda_j}} \frac{n^{r_l}}{\Pi_n(r_l)} \frac{\bC_{nm}}{n^{\lambda_{j}} (\log n)^{\nu_{j}}},
\end{equation*}
where the limit is in the almost sure as well as  $L^2$ sense. Thus, from~\eqref{eq: induction not lead}, using the fact that $\bpi^{(j)}$ is the left eigenvector of $R^{(j)}$ for the eigenvalue $\lambda_j$, we have,
\begin{equation} \label{eq: vj}
\frac{\bC^{(j)}_{Nk}}{N^{\lambda_{j}} (\log N)^{\nu_{j}}} = \frac{\bC_{Nl}}{N^{\lambda_{j}} (\log N)^{\nu_{j}}} \to \frac1{\lambda_j - r_l} \sum_{m=1}^{k - 1} \bpi^{(j)}_{m} r^{(j)}_{mk} V_j = \bpi^{(j)}_{k} V_j \ \text{almost surely and in} \ L^2.
\end{equation}
This completes the proof of the induction step and the proof of the Theorem is complete.
\end{proof}
\end{section}

\begin{section}{Three color urns} \label{sec: three col}
We now specialize to three color urns. The replacement matrix is then
\begin{equation} \label{eq: matrix three}
R =
\begin{pmatrix}
r_{11} &r_{12} &r_{13}\\
0 &r_{22} &r_{23}\\
0 &0 &1
\end{pmatrix}.
\end{equation}
We assume that the entries are non-negative, each row sum is one and~\eqref{eq: unique arrangement} holds. The latter is equivalent to assuming $r_{11}<1$, $r_{22}<1$; and  $r_{12}>0$, whenever $r_{11} = r_{22}$.

This three color urn model has already been considered
in~\cite{flajolet:dumas:puyhaubert:2006}, who further assumed that $r_{11}> 0$, $r_{12}>0 $ and $r_{22}> 0$. Under these  assumptions, they established the weak convergence of appropriately scaled $\bC_n$ and obtained the limit distributions \citep[cf.][Propositions 25 and 26]{flajolet:dumas:puyhaubert:2006}.

In contrast, we have established the almost sure convergence of scaled $\bC_n$. We restate our result as applicable to the three color urn.

\begin{cor} \label{cor: three}
Suppose we have a three color urn model with triangular replacement matrix $R$ given by~\eqref{eq: matrix three} with non-negative entries and each row sum one. Assume that $r_{11}<1$, $r_{22}<1$; and  $r_{12}>0$, whenever $r_{11} = r_{22}$. Then there exists nondegenerate random variables
$V_1$, $V_2$ and $V_3$ such that
\begin{enumerate}
\item $\bC_{n3} / n \to 1.$

\item If $r_{11}=0$, then $\bC_{n1}$ stays unchanged at $\bC_{01}$.\\
If $r_{11}>0$, then $\bC_{n1} / n^{r_{11}}\to V_1.$

\item If $r_{22} > r_{11}$, then $\bC_{n2} / n^{r_{22}} \to V_2.$

\item If $r_{22} = r_{11}$ and $r_{12} > 0$, then $\bC_{n2} / (n^{r_{22}} \log n) \to r_{12} V_1.$

\item If $r_{22} < r_{11}$ and $r_{12} > 0$, then $\bC_{n2} / n^{r_{11}} \to r_{12} V_1 / (r_{11} - r_{22}).$\\
If $0 < r_{22} < r_{11}$ and $r_{12} = 0$, then $\bC_{n2} / n^{r_{22}} \to V_3.$

\item If $r_{12} = r_{22} = 0$, then $\bC_{n2}$ stays unchanged at $\bC_{02}$.
\end{enumerate}

The convergence of all the above random variables is almost sure as well as in $L^2$.
\end{cor}

Three color urn models with reducible and block triangular balanced replacement matrices were considered in~\cite{bose:dasgupta:maulik:2008}. They established almost sure convergence of appropriately scaled individual color counts as well as weak/strong limits of \textit{linear combinations} $\bC_n\bzeta$ for suitable vectors $\bzeta$ obtained from the Jordan decomposition  of $R$.

Armed with the strong laws obtained from Corollary~\ref{cor: three}, we can now extend the results in~\cite{bose:dasgupta:maulik:2008} to the case of three color urn models with triangular replacement matrices. Observe that $\bxi_1 = (1,0,0)'$ and $\bxi_3 = (1,1,1)'$ are always right eigenvectors of $R$ with respect to the eigenvalues $r_{11}$ and $r_{33}$ respectively. Clearly,  $ \bC_n \bxi_3/ (n+1) = 1$ for all $n$. Also, since $ \bC_n \bxi_1 = \bC_{n1}$, its limiting behavior is given in Corollary~\ref{cor: three}~(ii).

Now observe that if $r_{11} \neq r_{22}$, then $R$ has a right eigenvector $\bxi_2$ with respect to the eigenvalue $r_{22}$, given by $\bxi_2 = (r_{12}, (r_{22} - r_{11}), 0)'$. If $r_{22} > r_{11}$, then, from Corollary~\ref{cor: three}~(ii)~and~(iii), we have
$$ \bC_n \bxi_2 / n^{r_{22}} \to (r_{22} - r_{11}) V_2\ \text{almost surely as well as in} \ L^2,$$
since the contribution of $\bC_{n1}$ is of smaller order.

If $r_{22} < r_{11}$ and $r_{12} = 0$, then, observe that $\bC_n \bxi_2 = (r_{22}-r_{11}) \bC_{n2}$. If we further have $r_{22} > 0$, then from Corollary~\ref{cor: three}~(v), we get
$$\bC_n \bxi_2 / n^{r_{22}} \to (r_{22} - r_{11}) V_3 \ \text{almost surely as well as in} \ L^2.$$ But, if we have $r_{12} = r_{22} = 0$, then $\bC_n \bxi_2$ remains constant at $\bC_0 \bxi_2$.

If $0 = r_{22} < r_{11}$ and $r_{12} > 0$, then observe that $\bxi_2$ being an eigenvector of $R$ with respect to the eigenvalue $r_{22} = 0$, $R \bxi_2$ becomes a null vector. Also, if, for $j=1, 2, 3$, the $j$-th color appears in the $n$-th draw, $\bC_n \bxi_2$ increases by an amount which is the $j$-th coordinate of $R \bxi_2$, namely $0$. Thus $\bC_n \bxi_2$ remains constant at $\bC_0 \bxi_2$.

The situation becomes interesting when $0 < r_{22} < r_{11}$ and $r_{12} > 0$. Note that in this case, we have, from Corollary~\ref{cor: three}~(ii)~and~(v), that $\bC_n \bxi_2 / n^{r_{11}} \to 0$ almost surely, as well as in $L^2$. We summarize the asymptotic behavior of $\bC_n \bxi_2$ in this case in the following proposition.
\begin{prop}
Suppose we have a three color urn model with triangular replacement matrix $R$ given by~\eqref{eq: matrix three} with non-negative entries and each row sum one. Assume that $0 < r_{22} < r_{11}$ and $r_{12} > 0$. Let $V_1$ be the almost sure limit of $\bC_{n1} / n^{r_{11}}$ obtained in Corollary~\ref{cor: three}~(ii). Then the following hold:
\begin{enumerate}
\item If $r_{22} < r_{11}/2$, then $\bC_n \bxi_2 / \sqrt{n^{r_{11}}} \Rightarrow N \left( 0, \frac{r_{12} r_{22}^2 (r_{12} + r_{11} - r_{22})}{r_{11} - 2 r_{22}} V_1 \right)$.

\item If $r_{22} = r_{11}/2$, then $\bC_n \bxi_2 / \sqrt{n^{r_{11}} \log n} \Rightarrow N \left( 0, r_{12} r_{22}^2 (r_{12} + r_{11} - r_{22}) V_1 \right)$.

\item If $r_{22} > r_{11}/2$, then $\bC_n \bxi_2 / n^{r_{22}}$ converges almost surely and in $L^2$ to a nondegenerate random variable.
\end{enumerate}
\end{prop}
Note that here $V_1$ is a random variable and the above limits are to be interpreted as variance mixtures of normal distributions.
\begin{proof}
The proof is same as that of Theorem~3.1 of~\cite{bose:dasgupta:maulik:2008}. The limiting variance in~(i) above will be $r_{22}^2 V_1 \bpi \bxi_2^2 / (r_{11} - 2 r_{22})$, where $\bxi_2^2$ is a column vector with coordinates which are squares of those of $\bxi_2$ and $\bpi = (1, r_{12} / (r_{11} - r_{22}), 0)$ is a left
eigenvector of $R$ corresponding to $r_{11}$. The limiting variance in~(ii) above will be $r_{22}^2 V_1 \bpi \bxi_2^2$. A simplification in either case gives the result.
\end{proof}

However, if $r_{12}>0$ and $r_{11} = r_{22}$, then for this repeated eigenvalue, it can be checked that for any $\alpha$, $\bxi_2=(\alpha, \ 1/r_{12},\ 0)^\prime$ is a Jordan vector satisfying $R\bxi_2=\bxi_1+r_{11}\bxi_2$. Hence, from Corollary~\ref{cor: three}~(ii) and~(iv), we have
$$\bC_n \bxi_2 / (n^{r_{22}}\log n) \to V_1\ \text{almost surely as well as in} \ L^2,$$
where $V_1$ is the almost sure limit of $\bC_{n1} / n^{r_{11}}$ obtained in Corollary~\ref{cor: three}~(ii).
\end{section}


\end{document}